\documentclass[a4paper,11 pt,reqno]{amsart}
\usepackage{a4,amsbsy,amscd,amssymb,amstext,amsmath,mathtools,mathdots,latexsym,oldgerm,enumerate,verbatim,graphicx,xy,ytableau}

\makeatletter                                       
\def\l@subsection{\@tocline{2}{0pt}{2.5pc}{2.5pc}{}}
\makeatother                                        

\makeatletter                                                  
\def\chapter{\clearpage\thispagestyle{plain}\global\@topnum\z@ 
\@afterindenttrue \secdef\@chapter\@schapter}
\makeatother

\newtheorem{thmgl} {Theorem}    

\newtheorem{lemgl} {Lemma}
\newtheorem{corgl} {Corollary}

\theoremstyle{definition}

\newtheorem{rem} {Remark} [section]

\newtheorem{exa} [rem] {Example}

\newtheorem{remgl} {Remark}
\newtheorem{remsgl} [remgl]{Remarks}

\newcommand{\mf}{\mathfrak}
\newcommand{\mc}{\mathcal}
\newcommand{\mb}{\mathbb}

\newcommand{\nts}{\negthinspace}     
\newcommand{\Nts}{\nts\nts}

\newcommand{\ov}{\overline}
\newcommand{\un}{\underline}

\newcommand{\sm}{\setminus}         

\newcommand{\ot}{\otimes}           
\newcommand{\la}{\langle}
\newcommand{\ra}{\rangle}

\newcommand{\Hom}{{\rm Hom}}        

\newcommand{\Mat}{{\rm Mat}}

\newcommand{\Sym}{{\rm Sym}} 

\newcommand{\sgn}{{\rm sgn}}

\newcommand{\id}{{\rm id}}

\let\ttie\t
\newcommand{\tie}[1]{{\let\t\ttie \ttie#1}}
\renewcommand{\t}{\mf{t}}  

\newcommand{\GL}{{\rm GL}}

\newcommand{\SL}{{\rm SL}}

\newcommand{\e}{\epsilon}

\makeatletter
\def\vcdots{\vbox{\baselineskip4\p@ \lineskiplimit\z@
\kern3\p@\hbox{.}\hbox{.}\hbox{.}\Nts\nts\kern3\p@}}
\makeatother

\xyoption{all}

\begin{document}

\title{Bases for spaces of highest weight vectors in arbitrary characteristic}

\begin{abstract}
Let $k$ be an algebraically closed field of arbitrary characteristic.
First we give explicit bases for the highest weight vectors for the action of $\GL_r\times\GL_s$ on
the coordinate ring $k[\Mat_{rs}^m]$ of $m$-tuples of $r\times s$-matrices.
It turns out that this is done most conveniently by giving an explicit good $\GL_r\times\GL_s$-filtration on $k[\Mat_{rs}^m]$.
Then we deduce from this result explicit spanning sets of the $k[\Mat_n]^{\GL_n}$-modules of highest weight vectors in the coordinate ring $k[\Mat_n]$ under the conjugation action of $\GL_n$. 
\end{abstract}

\author[A.\ Dent]{Adam Dent}
\author[R.\ Tange]{Rudolf Tange}

\maketitle
\markright{\MakeUppercase{Bases for highest weight vectors}}

\section*{Introduction}\label{s.intro}
Let $k$ be  an algebraically closed field, let $\GL_n$ be the group of invertible $n\times n$ matrices with entries in $k$ and let $T_n$ and $U_n$ be the subgroups of diagonal matrices and of upper uni-triangular matrices respectively. The group $\GL_r\times\GL_s$ acts on the $k$-vector space $\Mat_{rs}^m$ of $m$-tuples of $r\times s$ matrices with entries in $k$ via $((A,B)\cdot\un X)_i=AX_iB'$, where $\un X=(X_1,\ldots,X_m)\in\Mat_{rs}^m$ and $B'$ is the transpose of $B$, and on the coordinate ring $k[\Mat_{rs}^m]$ via $((A,B)\cdot f)(\un X)=f((A',B')\cdot\un X)=f((A'X_iB)_{1\le i\le m})$. For $(\mu,\lambda)$ a character of $T_r\times T_s$, the space of highest weight vectors will be denoted $k[\Mat_{rs}^m]_{(\mu,\lambda)}^{U_r\times U_s}$. It consists of the functions $f\in k[\Mat_{rs}^m]$ with $(A,B)\cdot f=f$ for all $(A,B)\in U_r\times U_s$ and $(A,B)\cdot f=\mu(A)\lambda(B)f$ for all $(A,B)\in T_r\times T_s$.

Our first goal in this paper is to give bases of the vector spaces $k[\Mat_{rs}^m]_{(\mu,\lambda)}^{U_r\times U_s}$. In \cite{T3} this was done under the assumption that $k$ is of characteristic $0$.
The method there was to reduce the problem via a few simple isomorphisms to certain results from the representation theory of the symmetric group which were originally due to J.~Donin. Although this method is rather straightforward, it is hard to generalise to arbitrary characteristic. In the present paper we solve the problem in arbitrary characteristic using results on bideterminants from the work of Kouwenhoven \cite{Kou} which is based on work of Clausen \cite{Clau}, \cite{Clau2}.
We introduce ``twisted bideterminants" to construct an explicit ``good" filtration and, in particular, give bases for the spaces of highest weight vectors in $k[\Mat_{rs}^m]$, see Theorem~\ref{thm.good_filtration} and its two corollaries in Section~\ref{s.several_matrices}. It turns out that these bases can also be obtained by dividing the basis elements from \cite[Thm.~4]{T3} by certain integers in the obvious $\mb Z$-form and then reducing mod $p$.

As an application we give in Section~\ref{s.conjugation_action} explicit finite homogeneous spanning sets of the $k[\Mat_n]^{\GL_n}$-modules of highest weight vectors in the coordinate ring $k[\Mat_n]$ under the conjugation action of $\GL_n$, see Theorem~\ref{thm.pullback_span_2} in in Section~\ref{s.conjugation_action}. Although this problem is difficult to tackle directly, we gave in \cite{T3} a method in arbitrary characteristic called ``transmutation" to reduce this problem to giving spanning sets for the vector spaces $k[\Mat_{rs}^m]_{(\mu,\lambda)}^{U_r\times U_s}$, see Theorem~\ref{thm.surjective_pullback} in the present paper. So the problem is reduced to the problem we solved in Section~\ref{s.several_matrices}.

\section{Preliminaries}\label{s.prelim}
The field $k$, the groups $\GL_n,U_n,T_n$, the variety $\Mat_{rs}^m$ and its coordinate ring $k[\Mat_{rs}^m]$ are as in the introduction.
Note that $k[\Mat_{rs}^m]$ is the polynomial algebra over $k$ in the variables $x(l)_{ij}$, $1\le l\le m$, $1\le i\le r$, $1\le j\le s$, where $x(l)_{ij}$ is the entry in the $i$-th row and $j$-column of the $l$-th matrix. If $m=1$ we write $x_{ij}$ instead of $x(1)_{ij}$. The $\GL_r\times\GL_s$-module $k[\Mat_{rs}^m]$ is multigraded by tuples of integers $\ge0$ (not necessarily partitions) of length $m$.
We denote the set of such tuples with coordinate sum $t$ by $\Sigma_{m,t}$ or just $\Sigma_t$. In this section we will only consider the $\GL_r\times\GL_s$-module $k[\Mat_{rs}]$, although we will
use the set $\Sigma_{m,t}$.


\subsection{Skew Young diagrams and tableaux}\label{s.skew_diagrams_and_tableaux}
In this section we introduce some combinatorics that we will need in Section~\ref{s.several_matrices} and which originates from \cite{JP}, \cite{Z1,Z2}, and \cite{Donin1,Donin2}.
In Section~\ref{s.bidets_skew_Schur_Specht} we discuss interpretations in terms of representation theory.

For $\lambda$ a partition of $n$ we denote the length of $\lambda$ by $l(\lambda)$ and its coordinate sum by $|\lambda|$.
We will identify each partition $\lambda$ with the corresponding Young diagram $\{(i,j)\,|\,1\le i\le l(\lambda),1\le j\le\lambda_i\}$. The $(i,j)\in\lambda$ are called the {\it boxes} or {\it cells} of $\lambda$.
More generally, if $\lambda,\mu$ are partitions with $\lambda\supseteq\mu$, then we denote the diagram $\lambda$ with the boxes of $\mu$ removed by $\lambda/\mu$ and call it the {\it skew Young diagram} associated to the pair $(\lambda,\mu)$. Of course the skew diagram $\lambda/\mu$ does not determine $\lambda$ and $\mu$. For a skew diagram $E$, we will denote the transpose by $E'$ and the number of boxes by $|E|$. The group of permutations of the boxes of $E$ will be denoted by $\Sym(E)$, and the column stabliser of $E$ in $\Sym(E)$, that is, the product of the groups of permutations of each column of $E$, will be denoted by $C_E$. By {\it diagram mapping} we mean a bijection between two diagrams as subsets of $\mb N\times\mb N$.

Let $E$ be a skew diagram with $t$ boxes. A {\it skew tableau} of shape $E$ is a mapping $T:E\to \mb N=\{1,2,\ldots\}$.
A skew tableau of shape $E$ is called {\it ordered} if its entries are weakly increasing along rows and weakly increasing along columns, and it is called {\it semi-standard} if its entries are weakly increasing along rows and strictly increasing along columns. It is called a {\it $t$-tableau} if its entries are the numbers $1,\ldots,t$ (so the entries must be distinct). A $t$-tableau whose entries are strictly increasing along both columns and rows is called {\it standard}.
If $m$ is the biggest integer occurring in a tableau $T$, then the {\it weight} of $T$ is the $m$-tuple whose $i$-th component is the number of occurrences of $i$ in $T$. Sometimes we will also consider
the weight of $T$ as an $m'$-tuple for some $m'\ge m$ by extending it with zeros.

For a skew shape $E$ with $t$ boxes, we define the {\it canonical} skew tableau $S_E$ by filling the boxes in the $i$-th row with $i$'s, and we define the tableau $T_E$ by filling in the numbers $1,\ldots,t$ row by row from left to right and top to bottom. So $S_E$ is semi-standard, and $T_E$ is a $t$-tableau which is standard. The {\it standard enumeration} of a tableau $T$ of shape $E$ is the $t$-tuple obtained from $T$ by reading its entries row by row from left to right and top to bottom.

Let $\mu$ be the tuple of row lengths of $E$, i.e. the weight of $S_E$. Let $S$ be a tableau of shape $F$ and weight $\mu$. If if $S=S_E\circ\alpha$ for some diagram mapping $\alpha:F\to E$,
then we say that $\alpha$ {\it represents} $S$. We call $S$ {\it ($E$-)special} if it is semi-standard and has a representative $\alpha:F\to E$
such that for any $a,b\in F$, if $\alpha(b)$ occurs strictly below $\alpha(a)$ in the same column, then $b$ occurs in a strictly lower row than $a$. We call a diagram mapping $\alpha:F\to E$ {\it admissible} if for $\alpha(b)$ strictly below $\alpha(a)$ in the same column, $b$ occurs in a strictly lower row than $a$ {\it and} in a column to the left of $a$ or in the same column.

Define two orderings $\le$ and $\preceq$ on $\mb N\times\mb N$ as follows: $(p,q)\le(r,s)$ if and only if $p\le r$ and $q\le s$, and $(p,q)\preceq(r,s)$ if and only if $p<r$ or ($p=r$ and $q\ge s$).
Note that $\preceq$ is a linear ordering. Recall that skew Young diagrams are by definition subsets of $\mb N\times\mb N$.
A diagram mapping $\alpha:F\to E$ is called {\it special} if $\alpha:(F,\le)\to(E,\preceq)$ and  $\alpha^{-1}:(E,\le)\to(F,\preceq)$ are order preserving.
So $\alpha$ is special if and only if $\alpha^{-1}$ is special.
If $\alpha$ is special, then $S_E\circ\alpha$ is special semi-standard and $\alpha$ is admissible, see \cite[p155-159]{Z1}.
Furthermore, by \cite[Thm.~3]{T3} every special semi-standard tableau has a unique special representative.

\begin{exa} 
Let $F=(2,2)$ and $E=(3,2)/(1)$ be skew diagrams. Since each has four boxes, we can construct a diagram mapping between the two shapes.
Give $F$ the standard enumeration. We now define a diagram mapping $\alpha_1:F\to E$ by numbering the boxes of $E$:
$a\in F$ is mapped by $\alpha_1$ to the box of $E$ which has the same number.
$$\begin{ytableau}1&2\\3&4\end{ytableau}
\hspace{7mm}\overset{\alpha_1}{\longrightarrow}\hspace{7mm}
\begin{ytableau}\none&1&2\\3&4\end{ytableau}$$
\smallskip
Then $\alpha_1$ is not admissible, since $4$ is below $1$ in the same column of $E$, but it occurs in a column strictly to the right of $1$ in $F$.
We now form the canonical tableau $S_E$ on $E$ and pull this numbering back to $F$ via $\alpha_1$ to obtain the tableau $S=S_E\circ\alpha_1$:
$$S=\begin{ytableau}1&1\\2&2\end{ytableau}
\hspace{7mm}\overset{\alpha_1}{\longrightarrow}\hspace{7mm}
\begin{ytableau}\none&1&1\\2&2\end{ytableau}\ =S_E$$
\smallskip
Clearly $S$ is semi-standard and for all $a,b\in F$, $b$ occurs in a strictly lower row than $a$ whenever $\alpha_1(b)$ occurs strictly below $\alpha_1(a)$ in the same column.
So $S$ is $E$-special semi-standard.
Now define $\alpha_2,\alpha_3:F\to E$ by:
$$\begin{ytableau}1&2\\3&4\end{ytableau}
\hspace{3mm}\overset{\alpha_2}{\longrightarrow}\hspace{3mm}
\begin{ytableau}\none&1&2\\4&3\end{ytableau}
\text{\hspace{5mm}and\hspace{5mm}}
\begin{ytableau}1&2\\3&4\end{ytableau}
\hspace{3mm}\overset{\alpha_3}{\longrightarrow}\hspace{3mm}
\begin{ytableau}\none&2&1\\4&3\end{ytableau}$$
\smallskip
Then $S_E\circ\alpha_2=S_E\circ\alpha_3=S$, $\alpha_2$ is admissible, but not special, and $\alpha_3$ is special.
The inverse of $\alpha_3$ is also special and is therefore the unique special representative of the $F$-special semi-standard tableau $T=S_F\circ\alpha_3^{-1}$ on $E$:
$$T=S_F\circ\alpha_3^{-1}=\hspace{1mm}\begin{ytableau}\none&1&1\\2&2\end{ytableau}
\hspace{7mm}\overset{\alpha_3^{-1}}{\longrightarrow}\hspace{5mm}
\begin{ytableau}1&1\\2&2\end{ytableau}\ =S_F$$
\smallskip
Besides $T$ there is one other semi-standard tableau $\tilde T$ of shape $E$ and weight $(2,2)$:
$$\tilde T=\hspace{5mm}\begin{ytableau}\none&1&2\\1&2\end{ytableau}\hspace{3mm}\overset{\beta}{\longrightarrow}\hspace{3mm}
\begin{ytableau}1&1\\2&2\end{ytableau}$$
This tableau is not $F$-special: if $\beta:E\to F$ is a diagram mapping with $\tilde T=S_F\circ\beta$ and $b$ is the rightmost box in the top row of $E$, 
then there must be a box $a$ of $E$ such that $\beta(a)$ is directly above $\beta(b)$ in the same column in $F$, but $a$ cannot occur in a higher row than $b$ in $E$.
\end{exa}

From now on we will always expect representatives of special semi-standard tableaux to be admissible. 

What we will be using in the proof of our main result Theorem~\ref{thm.good_filtration} is a refinement of the above combinatorics.
We need to cut $F$ and $E$ into pieces labelled by certain integers and then we work with certain diagram mappings $\alpha$ which map each piece of $F$ into the piece of $E$ of the same label.
We then apply the above combinatorics to the restrictions of $\alpha$ to these pieces.
Now let $E$ and $F$ be arbitrary skew diagrams each with $t$ boxes. Let $P$ and $Q$ be ordered tableaux of shapes $E$ and $F$, both of weight $\nu\in\Sigma_t$. Then a diagram mapping $\alpha:F\to E$
with $P\circ\alpha=Q$ determines an $m$-tuple of tableaux $(S_{P^{-1}(1)}\circ\alpha_1,\ldots,S_{P^{-1}(m)}\circ\alpha_m)$ (*),
where $\alpha_i:Q^{-1}(i)\to P^{-1}(i)$ is the restriction of $\alpha$ to $Q^{-1}(i)$.
We will say that $\alpha$ {\it represents} (*). Notice that the $m$-tuples (*), for varying $\alpha$, all have the same tuple of shapes and the same tuple of weights.
We express this by saying that the tuple of tableaux has {\it shapes determined by $Q$ and weights determined by $P$}.
When the tableaux $S_{P^{-1}(i)}\circ\alpha_i$ are special semi-standard, we require the $\alpha_i$ to be admissible.
For more detail see \cite[Sect.~3]{T3}, or \cite{Z1}, \cite{Z2} where special diagram mappings are defined as ``pictures".
\begin{exa}
Take $F=(4,4,3)/(1)$ and $E=(4,3,3)$ be skew diagrams, take $\nu=(4,6)$ and define $Q$ and $P$ as indicated below.
$$Q=\hspace{3mm}\ytableausetup{centertableaux}\begin{ytableau}\none&1&1&2\\1&1&2&2\\2&2&2\end{ytableau}
\hspace{7mm}\overset{\alpha}{\longrightarrow}\hspace{7mm}
\begin{ytableau}1&1&2&2\\1&1&2\\2&2&2\end{ytableau}\hspace{3mm}=P$$
Then $\alpha_1$ goes between the ``$1$-pieces" of $Q$ and $P$ and $\alpha_2$ goes between the ``$2$-pieces" of $Q$ and $P$.
We also indicate the canonical numberings on the pieces of $E$ and certain special semi-standard numberings on the pieces of $F$
which can be obtained by pulling back the canonical numberings along suitable $\alpha_i$.
$$\begin{ytableau}\none&1&1\\2&2\end{ytableau}
\hspace{3mm}\overset{\alpha_1}{\longrightarrow}\hspace{3mm}
\begin{ytableau}1&1\\2&2\end{ytableau}
\text{\hspace{5mm}and\ }
\begin{ytableau}\none&\none&\none&1\\\none&\none&1&2\\3&3&3\end{ytableau}
\hspace{3mm}\overset{\alpha_2}{\longrightarrow}
\begin{ytableau}\none&\none&1&1\\\none&\none&2\\3&3&3\end{ytableau}$$
The tableau $\hspace{3mm}\begin{ytableau}\none&\none&\none&1\\\none&\none&2&3\\1&3&3\end{ytableau}\vspace{3mm}\hspace{3mm}$
is the only other $P^{-1}(2)$-special semi-standard tableau of shape $Q^{-1}(2)$ and weight $(2,1,3)$.
\end{exa}

\subsection{Bideterminants and skew Schur and Specht modules}\label{s.bidets_skew_Schur_Specht}
In this section we will review some facts from the representation theory of the general linear group as well as the symmetric group.
The representation theory of the symmetric group will not be used in this paper, but it may help to understand the combinatorics we use.
It was also used in \cite[Thm.~4]{T3} to obtain a version in characteristic $0$ of Corollary~2 to Theorem~\ref{thm.good_filtration} from the present paper.

Let $E$ be a skew diagram with $t$ boxes.
Let $S$ and $T$ be tableaux of shape $E$, $S$ with entries $\le r$ and $T$ with entries $\le s$. Then we define the {\it bideterminant} $(S\,|\,T)\in k[\Mat_{rs}]$ by
$$(S\,|\,T)=\prod_{i=1}^n{\rm det}\big((x_{S(a),T(b)})_{a,b\in E^i}\big),$$
where $E^i$ is the $i$-th column of $E$ and $n$ is the number of columns in $E$. Note that we have
\begin{equation*}
(S\,|\,T)=\sum_{\pi\in C_E}{\rm sgn}(\pi)\prod_{a\in E}x_{S(\pi(a)),\,T(a)}=\sum_{\pi\in C_E}{\rm sgn}(\pi)\prod_{a\in E}x_{S(a),\,T(\pi(a))}\,,
\end{equation*}
where $C_E\le\Sym(E)$ is the column stabiliser of $E$.

As is well-known, the elements $(S\,|\,T)$, $S$ standard with entries $\le r$ and $T$ standard with entries $\le s$ form a basis of $k[\Mat_{rs}]$, see \cite{DKR}. In fact one can use bideterminants to construct
explicit ``good" filtrations of $k[\Mat_{rs}]$ as a $\GL_r\times\GL_s$-module, see \cite{DeCEP}.

\begin{exa}
Let $E=(3,2)/(1)$, $S=\begin{ytableau}\none&1&2\\1&2\end{ytableau}$\,, $T=\begin{ytableau}\none&2&3\\1&3\end{ytableau}$\,. Then $(S\,|\,T)\in k[\Mat_{2,3}]$ is given by
$$(S\,|\,T)\left(\begin{matrix}a&b&c\\d&e&f\end{matrix}\right)=a.\det\left(\begin{matrix}b&c\\e&f\end{matrix}\right).f=abf^2-acef\,.$$
\end{exa}

The {\it skew Schur module} associated to a shape $E$, denoted by $\nabla_{\GL_r}(E)$, is the span in $k[\Mat_{rs}]$, $s\ge$ the number of rows of $E$, of all the bideterminants $(S\,|\,S_E)$ where $S$ is a tableau of shape $E$ and with entries $\le r$. The skew Schur module $\nabla_{\GL_r}(E)$ will be nonzero if and only if $r$ is $\ge$ the length of each column of $E$. It can easily be seen that $\nabla_{\GL_r}(E)$ is $\GL_r$-stable, and it is well-known that the set of bideterminants $(S\,|\,S_E)$ with $S$ as above and in addition semi-standard form a basis. Note that if $E$ is an ordinary Young tableau then $\nabla_{\GL_r}(E)$ is the Schur (or induced) module associated to it.
The {\it co-Schur} or {\it Weyl module} $\Delta_{\GL_r}(E)$ associated to a shape $E$ can be defined as the contravariant dual $\nabla_{\GL_r}(E)^\circ$ which is the dual of the vector space $\nabla_{\GL_r}(E)$ with $\GL_r$ acting via the transpose: $(g\cdot f)(v)=f(g'\cdot v)$, where $g'$ is the transpose of $g\in\GL_r$. If $E$, $F$, $\mu$ are as in Section~\ref{s.skew_diagrams_and_tableaux}, then the number of special semi-standard tableaux of shape $F$ and weight $\mu$ is equal to the dimension of $\Hom_{\GL_r}(\Delta_{\GL_r}(F),\nabla_{\GL_r}(E))$ whenever $r$ is $\ge$ the number of rows of $E$ or $\ge$ the number of rows of $F$. This can be seen by reducing to the case that $k$ has characteristic $0$ using \cite[Prop.~II.4.13]{Jan}, and then using standard properties of skew Schur functions. 
For more details we refer to \cite{ABW} and \cite{Kou}.

Drop for the moment the assumption that $k$ is algebraically closed. The {\it skew Specht module} $S(E)=S_t(E)=S_{t,k}(E)$ for the group algebra $A=A_{t,k}=k\Sym_t$ of the symmetric group $\Sym_t$ on $\{1,\ldots,t\}$ is defined just as in the case of an ordinary Young diagram: $S(E)=Ae_1e_2$, where $e_1$ is the column anti-symmetriser of $T_E$ and $e_2$ is the row symmetriser of $T_E$. The module $M(E)=M_{t,k}(E)=Ae_2$ is called the {\it permutation module} associated to $E$. One can also define $S_{t,k}(E)$ as the weight space $\nabla_{\GL_r}(E)_{1^t}$ for any $r\ge t$. Note that this weight space is indeed stable under $\Sym_t\le\GL_r$.
Now assume $k$ is of characteristic $0$ and let $E$, $F$, $\mu$ be as in Section~\ref{s.skew_diagrams_and_tableaux}. Then the number of special semi-standard tableaux of shape $F$ and weight $\mu$ is equal to the dimension of $\Hom_{\Sym_t}(S(E),S(F))\cong\big(S(E)\ot S(F)\big)^{\Sym_t}\cong\big(S(E)\ot S(F)\big)_{\Sym_t}$, where $N_{\Sym_t}$ denotes the space of coinvariants of an $A$-module $N$, i.e. the quotient of $N$ by the span of the elements $x-g\cdot x$, $x\in N$, $g\in\Sym_t$.

Assume $r=r_1+\cdots+r_m$ for certain integers $r_i>0$. For each  $\nu\in\Sigma_{m,t}$ there exists a $\big(\prod_{i=1}^m\GL_{r_i}\big)$-module filtration of the piece of multidegree $\nu$ of $\nabla_{\GL_r}(E)$ with sections in some order isomorphic to the modules $\bigotimes_{i=1}^m\nabla_{\GL_{r_i}}\big(P^{-1}(i)\big)$, $P$ an ordered tableau of shape $E$ and weight $\nu$. Here we can omit the $P$'s for which $P^{-1}(i)$ has a column of length $>r_i$ for some $i$.
See \cite[Thm.~II.4.11]{ABW} or \cite[Thm.~1.4]{Kou} and Remark~2 after it.
Now assume that $k$ has characteristic $0$, let $\nu\in\Sigma_{m,t}$ and let $\Sym_\nu\le\Sym_t$ be the Young subgroup associated to $\nu$. Then we have an isomorphism $S_t(E)\cong\bigoplus_P\bigotimes_{i=1}^mS_{\nu_i}(P^{-1}(i))$ of $\Sym_\nu$-modules, where the sum is over all ordered tableau $P$ of shape $E$ and weight $\nu$.

\begin{remsgl}\label{rems.bidet}
1.\ One can of course also define $\nabla_{\GL_r}(E)$ as the span in $k[\Mat_{sr}]$ of all the bideterminants $(S_E\,|\,T)$ where $T$ is a tableau of shape $E$ and with entries $\le r$.
Then the action of $\GL_r$ comes from the right multiplication rather than from the left multiplication.\\
2.\ Let $\lambda$ and $\mu$ be partitions with $\mu\subseteq\lambda$. Let $r,r_1,s$ be integers $\ge0$ with $r_1,s\ge l(\lambda)$ and $r_1\ge l(\mu)+r$ and put $r'=r_1-r$.
We embed $\GL_{r'}\times\GL_r$ in $\GL_{r_1}$ such that $\GL_r$ fixes the first $r'$ basis vectors. 
Then one can embed $\nabla_{\GL_r}(\lambda/\mu)$ as a $\GL_r$-submodule in $\nabla_{\GL_{r_1}}(\lambda)$. Indeed one can deduce from \cite{Don} that $\nabla_{\GL_r}(\lambda/\mu)\cong\Hom_{\GL_{r'}}(\Delta_{\GL_{r'}}(\mu),\nabla_{\GL_{r_1}}(\lambda))\cong\nabla_{\GL_{r_1}}(\lambda)^{U_{r'}}_\mu$, where $\mu$ is considered as a weight for $T_{r'}$.
One can also construct an explicit isomorphism as follows. Let $E\in\Mat_{r's}$ be the matrix whose first $\min(r',s)$ rows are those of the $s\times s$ identity matrix followed by $r'-s$ zero rows if $r'>s$.
Then the comorphism of the morphism $A\mapsto\big[\begin{smallmatrix}E\\A\end{smallmatrix}\big]:\Mat_{rs}\to\Mat_{r_1s}$ maps $\nabla_{\GL_{r_1}}(\lambda)^{U_{r'}}$ isomorphically onto $\nabla_{\GL_r}(\lambda/\mu)$. 
Combinatorially this is easy to understand: $\nabla_{\GL_{r_1}}(\lambda)^{U_{r'}}$ has a basis labelled by semi-standard tableaux of shape $\lambda$
with entries $\le r_1$ in which the entries $\le r'$ occupy the boxes of $\mu$ and form the canonical tableau $S_\mu$. These tableaux are clearly in one-one correspondence with the semi-standard tableaux of shape $\lambda/\mu$ with entries $\le r$: just remove the $\mu$-part and subtract $r'$ from the entries of the resulting tableau of shape $\lambda/\mu$.\\
3.\ The present paper uses bideterminants and the representation theory of the general linear group.
We compare it with the approach in \cite[Sect.~3]{T3} which uses the representation theory of the symmetric group.

Let $\lambda$ and $\mu$ be partitions of $t$ with $l(\mu)\le r$ and $l(\lambda)\le s$ and
let $\nu\in\Sigma_t$.
Then we have that the space of coinvariants $\big(M(\mu)\ot M(\lambda)\big)_{\Sym_\nu}$ is isomorphic to the piece of multidegree $\nu$
of the weight space $k[\Mat_{rs}^m]_{(\mu,\lambda)}$, see \cite[proof of Thm.~3]{T3} and also \cite[Thm 3.7]{Clau}.
The permutation module $M(\mu)$ can be identified with the weight space $k[\Mat_{t\,r}]_{(1^t,\mu)}$ and similar for $M(\lambda)$.
If $k$ has characteristic $0$, then $\big(S(\mu)\ot S(\lambda)\big)_{\Sym_\nu}$ embeds in $\big(M(\mu)\ot M(\lambda)\big)_{\Sym_\nu}$.
In \cite[Sect.~3]{T3} we worked with $S(\mu)$ and $S(\lambda)$ which can be thought of as spanned by bideterminants
$(T\,|\,S_\mu)$, $T$ a $t$-tableau of shape $\mu$, and $(T\,|\,S_\lambda)$, $T$ a $t$-tableau of shape $\lambda$.
Actually, we mostly worked with skew versions of $S(\mu)$ and $S(\lambda)$. Only after \cite[Prop.~3]{T3} we passed to coinvariants.
In the present paper we work entirely inside the space of coinvariants which is the degree $\nu$ piece of $k[\Mat_{rs}^m]_{(\mu,\lambda)}$.
This means that $t$-tableaux play almost no role, they are ``replaced" by diagram mappings $\alpha:\mu\to\lambda$.
The canonical tableaux $S_\mu$ and $S_\lambda$ are now arbitrary tableaux $S$ and $T$ of shape $\mu$ and $\lambda$ and we work with twisted bideterminants $(S\,|_\alpha\,T)$.
\end{remsgl}

\section{The action of $\GL_r\times\GL_s$ on several $r\times s$-matrices}\label{s.several_matrices}
Let $\lambda,\mu$ be partitions of $t$ with $l(\mu)\le r$ and $l(\lambda)\le s$, let $P,Q$ ordered tableaux of shapes $\lambda$ and $\mu$, both of weight $\nu\in\Sigma_t$ and
$\alpha:\mu\to\lambda$ a diagram mapping such that $P\circ\alpha=Q$.
Define $u_{P,Q,\alpha}\in k[\Mat_{rs}^m]^{U_r\times U_s}_{(\mu,\lambda)}$ by
\begin{equation}\label{eq.u}
u_{P,Q,\alpha}=\sum_{\pi\in C_\mu, \sigma\in C_\lambda}{\rm sgn}(\pi){\rm sgn}(\sigma)\prod_{a\in\mu}x(Q(a))_{\pi(a)_1,\,\sigma(\alpha(a))_1}\,,
\end{equation}
where $b_1$ is the row index of a box $b$. It was proved in \cite[Thm.~4]{T3} that for suitable $(\nu,P,Q,\alpha)$ these elements form a basis of
the vector space $k[\Mat_{rs}^m]^{U_r\times U_s}_{(\mu,\lambda)}$ when $k$ has characteristic $0$.\footnote{In \cite{T3} we used the inverse rather than the transpose for the action of $\GL_r$,
which explains why we have $\pi(a)_1$ here rather than $r-\pi(a)_1+1$.} 
More generally, we can consider for $E$ and $F$ skew shapes with $t$ boxes, $P,Q$ tableaux of shapes $E$ and $F$, both of weight $\nu\in\Sigma_t$,
$\alpha:F\to E$ a diagram mapping such that $P\circ\alpha=Q$, $S$ a tableau of shape $F$ with entries $\le r$ and $T$ a tableau of shape $E$ with entries $\le s$ the sum
\begin{equation}\label{eq.twisted_bidet0}
\sum_{(\pi, \sigma)\in C_F\times C_E}{\rm sgn}(\pi){\rm sgn}(\sigma)\prod_{a\in F}x(Q(a))_{S(\pi(a)),\,T(\sigma(\alpha(a)))}\,.
\end{equation}
Note that we obtain \eqref{eq.u} from \eqref{eq.twisted_bidet0} by taking $S$ and $T$ the canonical tableaux $S_F$ and $S_E$.

We will now show that \eqref{eq.twisted_bidet0} is in $\mb Z[\Mat_{rs}^m]=\mb Z[(x(l)_{ij})_{lij}]$ divisible by the order of the subgroup
$$C_{P,Q,\alpha}=\{(\tau,\rho)\in C_F(Q)\times C_E(P)\,|\,\alpha\circ\tau\circ\alpha^{-1}=\rho\}$$
of $C_F\times C_E$, where $C_F(Q)$ is the stabiliser of $Q$ in $C_F$ and $C_E(P)$ is defined similarly.
Note that
$$C_{P,Q,\alpha}\cong C_F(Q)\cap\alpha^{-1}C_E(P)\alpha\le\Sym(F)\cong\prod_{i=1}^mC_{Q^{-1}(i)}\cap\alpha_i^{-1}C_{P^{-1}(i)}\alpha_i$$ and that
$$C_{P,Q,\alpha}\cong \alpha C_F(Q)\alpha^{-1}\cap C_E(P)\le\Sym(E)\cong\prod_{i=1}^m\alpha_iC_{Q^{-1}(i)}\alpha_i^{-1}\cap C_{P^{-1}(i)}\,.$$
In each of the two lines above one may omit ``$(Q)$" in $C_F(Q)$ or ``$(P)$" in $C_E(P)$, but not both.
\begin{lemgl}
Each summand in \eqref{eq.twisted_bidet0} only depends on the left coset of $(\pi,\sigma)$ modulo $C_{P,Q,\alpha}$.
\end{lemgl}
\begin{proof}
Let $(\pi_1,\sigma_1),(\pi_2,\sigma_2)\in C_F\times C_E$ and suppose $(\pi_2,\sigma_2)=(\pi_1\circ\tau,\sigma_1\circ\rho)$ for some $(\tau,\rho)\in C_{P,Q,\alpha}$. Then $\sgn(\pi_1)\sgn(\sigma_1)=\sgn(\pi_2)\sgn(\sigma_2)$, since $\sgn(\tau)=\sgn(\rho)$. Furthermore,
\begin{align*}
\underset{a\in F}{\prod}x(Q( a))_{S(\pi_2(a)),\,T(\sigma_2(\alpha( a)))}&=\underset{a\in F}{\prod}x( Q( a))_{S(\pi_1(\tau( a))),\,T(\sigma_1(\rho(\alpha( a))))}\\
&=\underset{a\in F}{\prod}x( Q( a))_{S((\tau( a))),\,T(\sigma_1(\alpha(\tau( a))))}\\
&=\underset{a\in F}{\prod}x(Q(\tau^{-1}( a)))_{S(\pi_1( a)),\,T(\sigma_1(\alpha( a)))}\\
&=\underset{a\in F}{\prod}x(Q(a))_{S(\pi_1( a)),\,T(\sigma_1(\alpha(a)))}.
\end{align*}
\end{proof}

We now define the {\it twisted bideterminant} $(S\,|^m_\alpha\,T)\in k[\Mat_{rs}^m]$ by
\begin{equation}\label{eq.twisted_bidet1}
(S\,|^m_\alpha\,T)=\sum_{(\pi, \sigma)}{\rm sgn}(\pi){\rm sgn}(\sigma)\prod_{a\in F}x(Q(a))_{S(\pi(a)),\,T(\sigma(\alpha(a)))}\,,
\end{equation}
where the sum is over a set of representatives of the left cosets of $C_{P,Q,\alpha}$ in $C_F\times C_E$.
To keep the notation manageable, we suppressed ($P$ and) $Q$ in $(S\,|^m_\alpha\,T)$.
Clearly, if $k$ has characteristic $0$, then $(S\,|^m_{\alpha}\,T)$ equals \eqref{eq.twisted_bidet0} divided by $|C_{P,Q,\alpha}|$.
Note that the product in \eqref{eq.twisted_bidet1} can also be written as $$\prod_{a\in E}x(P(a))_{S(\pi(\alpha^{-1}(a))),\,T(\sigma(a))}\,.$$

In case $m=1$, $P$ and $Q$ are constant equal to $1$ and they play no role. We then omit $P,Q$ and the superscript $m$ in our notation
and instead of $x(1)_{ij}$ we write $x_{ij}$. So

\begin{equation}\label{eq.twisted_bidet2}
(S\,|_\alpha\,T)=\sum_{(\pi, \sigma)}{\rm sgn}(\pi){\rm sgn}(\sigma)\prod_{a\in F}x_{S(\pi(a)),\,T(\sigma(\alpha(a)))}\,,
\end{equation}
where the sum is over a set of representatives of the left cosets of\\
$C_\alpha=\{(\tau,\rho)\in C_F\times C_E\,|\,\alpha\circ\tau\circ\alpha^{-1}=\rho\}$ in $C_F\times C_E$.
Note that if $m=1$, $E=F$ and $\alpha=\id$ we get the ordinary bideterminant.

\begin{remgl}\label{rem.twisted_bidet}
If $X$ is a set of representatives for the left cosets of $\alpha C_F(Q)\alpha^{-1}\cap C_E(P)$ in $C_E$, then $C_F\times X$ is a 
set of representatives for the left cosets of $C_{P,Q,\alpha}$ in $C_F\times C_E$.
If we concatenate all matrices in an $m$-tuple column-wise, then we obtain an isomorphism $k[\Mat_{rs}^m]\cong k[\Mat_{r,ms}]$
which maps $x(l)_{ij}$ to $x_{i,(l-1)s+j}$. Now we have
$$(S\,|^m_\alpha\,T)=\sum_{\sigma\in X}\sgn(\sigma)(S\,|\,T^{\alpha,\sigma})\,,$$
where $T^{\alpha,\sigma}(a)=T(\sigma(\alpha(a)))+(Q(a)-1)s$ for $a\in F$.
Of course we could also work with a set $\tilde X$ of representatives for the left cosets of $C_F(Q)\cap\alpha^{-1}C_E(P)\alpha$ in
$\tilde C_F=\alpha^{-1}C_E\alpha$. Then the above sum would be over $\sigma\in\tilde X$ with $T^{\alpha,\sigma}(a)=T(\alpha(\sigma(a)))+(Q(a)-1)s$ for $a\in F$.

Similarly, if $X$ is a set of representatives for the left cosets of $C_F(Q)\cap\alpha^{-1}C_E(P)\alpha$ in $C_F$, then $X\times C_E$ is a 
set of representatives for the left cosets of $C_{P,Q,\alpha}$ in $C_F\times C_E$.
If we concatenate all matrices in an $m$-tuple row-wise, then we obtain an isomorphism $k[\Mat_{rs}^m]\cong k[\Mat_{mr,s}]$ which
maps $x(l)_{ij}$ to $x_{(l-1)r+i,j}$. Then we have
$$(S\,|^m_\alpha\,T)=\sum_{\pi\in X}\sgn(\pi)(S^{\alpha,\pi}\,|\,T)\,,$$
where $S^{\alpha,\pi}(a)=S(\pi(\alpha^{-1}(a)))+(P(a)-1)r$ for $a\in E$.
With $\tilde X$ a set of representatives for the left cosets of $\alpha C_F(Q)\alpha^{-1}\cap C_E(P)$ in $\tilde C_E=\alpha C_F\alpha^{-1}$, the above
sum would be over $\pi\in\tilde X$ with $S^{\alpha,\pi}(a)=S(\alpha^{-1}(\pi(a)))+(P(a)-1)r$ for $a\in E$.

In the case of the twisted bideterminants $(S\,|_\alpha\,T)$ for a single matrix $P$ and $Q$ play no role, so $C_F(Q)$ and $C_E(P)$ can
be replaced by $C_F$ and $C_E$, and in the definitions of $T^{\alpha,\sigma}$ and $S^{\alpha,\pi}$ the terms containing $Q$ or $P$ should be omitted.
The twisted bideterminants $(S\,|_\alpha\,T)$ are known as ``shuffle-products", and moving from the single matrix version of the first expression above to that of the second
is called ``overturn of the P-shuffle product onto the L-side", see \cite[Sect.~8-11]{Clau2}.
\end{remgl}
\begin{exa}
Take $F=(2,2)$ and $E=(2,1,1)$ and $Q$, $P$ and $\alpha$ as indicated below.
$$Q=\hspace{2mm}\begin{ytableau}1&2\\2&2\end{ytableau}\hspace{9mm}
\begin{ytableau}1&2\\3&4\end{ytableau}
\hspace{4mm}\overset{\alpha}{\longrightarrow}\hspace{4mm}
\begin{ytableau}1&3\\2\\4\end{ytableau}\hspace{9mm}
\begin{ytableau}1&2\\2\\2\end{ytableau}\hspace{2mm}=P$$
Here the second tableau of shape $F$ has been given the standard numbering and $\alpha$ maps each box of $F$ to the box of $E$ with the same number.
Clearly, $P\circ\alpha=Q$. Note that $\alpha_1:\begin{ytableau}1\end{ytableau}\hspace{2mm}{\longrightarrow}\hspace{2mm}\begin{ytableau}1\end{ytableau}\hspace{2mm}$ and
$\alpha_2:\begin{ytableau}\none&2\\3&4\end{ytableau}\hspace{2mm}{\longrightarrow}\hspace{2mm}\begin{ytableau}\none&3\\2\\4\end{ytableau}$ are special.

Now we consider certain twisted bideterminants in $k[\Mat_{23}^2]\cong k[\Mat_{43}]$. For $S$ a tableau of shape $F$ with entries $\le2$, $T$ a tableau of shape $E$ with entries $\le3$, and
$\alpha,P,Q$ as above we have $$(S\,|^2_\alpha\,T)=\left(\left.\begin{array}{ccc}s_{11}&s_{21}+2\\s_{12}+2\\s_{22}+2\end{array}\,\right\vert\ T\right)-
\left(\left.\begin{array}{ccc}s_{21}&s_{11}+2\\s_{12}+2\\s_{22}+2\end{array}\,\right\vert\ T\right)\,.$$
This can be seen by applying Remark~\ref{rem.twisted_bidet} to the set of representatives $X=\la(1,3)\ra\le C_F$ for the left cosets of $C_F(Q)\cap\alpha^{-1}C_E(P)\alpha=\la(2,4)\ra$ in $C_F$.
\end{exa}
\medskip

The coordinate ring $k[\Mat_{rs}^m]$ is $\mb N_0^m$-graded. Fix a multidegree $\nu\in\Sigma_t$.
Then one can construct a filtration $$M_1\supseteq M_2\supseteq\cdots\supseteq M_{q+1}=0$$
with sections isomorphic to $\nabla_{\GL_r}(\mu)\ot\nabla_{\GL_s}(\lambda)$, $\mu,\lambda$ suitable,
of the graded piece $M_1$ of degree $\nu$ of $k[\Mat_{rs}^m]$ as follows.
We use triples $(P,Q,\alpha)$ where $P$ and $Q$ are ordered tableaux of weight $\nu$ with shapes $\lambda$ of length $\le s$ and $\mu$ of length $\le r$ say, and
$\alpha:\mu\to\lambda$ is in a set of (admissible) representatives for the $m$-tuples of special semi-standard tableaux with shapes determined by $Q$ and weights determined by $P$.
See Section~\ref{s.skew_diagrams_and_tableaux}.

\begin{thmgl}\label{thm.good_filtration}
We can enumerate all the triples $(P,Q,\alpha)$ as above:
$$(P_1,Q_1,\alpha^1),(P_2,Q_2,\alpha^2),\ldots,(P_q,Q_q,\alpha^q),$$
$\lambda^i$ the shape of $P_i$, $\mu^i$ the shape of $Q_i$,
such that for all $i$ the span $M_i$ of all twisted bideterminants $(S\,|^m_{\alpha^j}\,T)$,
$j\ge i$, $S$ of shape $\mu^j$ with entries $\le r$, $T$ of shape $\lambda^j$ with entries $\le s$, is $\GL_r\times\GL_s$-stable and we have an isomorphism
$$(S\,|\, S_{\mu^i})\ot(T\,|\,S_{\lambda^i})\mapsto (S\,|^m_{\alpha^i}\,T)\text{\rm\ mod\ }M_{i+1}:\nabla_{\GL_r}(\mu^i)\ot\nabla_{\GL_s}(\lambda^i)\stackrel{\sim}{\to}M_i/M_{i+1}\,.$$
Furthermore, the twisted bideterminants $(S\,|^m_{\alpha^j}\,T)$, $1\le j\le q$, $S$ and $T$ as above and in addition semi-standard,
form a basis of the graded piece of degree $\nu$ of $k[\Mat_{rs}^m]$.
\end{thmgl}

\begin{proof}
We use the isomorphism $k[\Mat_{rs}^m]\cong k[\Mat_{mr,s}]$, see Remark~\ref{rem.twisted_bidet}. Let $t$ be an integer $\ge 0$.
We start with the well-known (descending) $\GL_{mr}\times\GL_s$-filtration of the piece of degree $t$ of $k[\Mat_{mr,s}]$ with sections isomorphic to
\begin{equation}\label{eq.sectioniso1}
\nabla_{\GL_{mr}}(\lambda^i)\ot\nabla_{\GL_s}(\lambda^i).
\end{equation}
Here the $\lambda^i$ are the partitions of $t$ of length $\le\min(mr,s)$. The isomorphisms to the sections of the filtration are given by
$$(S\,|\, S_{\lambda^i})\ot(T\,|\,S_{\lambda^i})\mapsto(S\,|\,T)\,\text{\ modulo the $(i+1)$-th filtration space}.$$
After restricting the left multiplication action to $\GL_r^m$ we can decompose the above filtration according to the multidegree in $\mb N_0^m$.
From now on we focus on the piece of multidegree $\nu\in\Sigma_t$.
By repeatedly applying \cite[Thm.~II.4.11]{ABW} (see also \cite[Thm.~1.4]{Kou} and Remark~2 after it) to $\nabla_{\GL_{mr}}(\lambda^i)$
we can refine the above filtration to a filtration with sections isomorphic to
\begin{equation}\label{eq.sectioniso2}
\Big(\bigotimes_{j=1}^m\nabla_{\GL_r}\big(P_i^{-1}(j)\big)\Big)\ot\nabla_{\GL_s}(\lambda^i).
\end{equation}
Here the $\lambda^i$ are suitably redefined, the $P_i$ go through all ordered tableaux of shape $\lambda^i$ with weight $\nu$, and the Levi $\GL_r^m$ acts on the first factor.
The section-isomorphism of \cite[Thm.~II.4.11]{ABW} is given by shifting the numbers in each tableau of shape $P_i^{-1}(j)$ by $(j-1)r$, so the result has its entries in $(j-1)r+\{1,\ldots,r\}$,
and then piecing the resulting tableaux of shapes $P_i^{-1}(j)$ together according to $P_i$ to a tableau of shape $\lambda^i$. 
Now we restrict the first factor of \eqref{eq.sectioniso2} to the diagonal copy of $\GL_r$ in $\GL_r^m$ and we have
\begin{equation}\label{eq.iso3}
\bigotimes_{j=1}^m\nabla_{\GL_r}\big(P_i^{-1}(j)\big)\cong\nabla_{\GL_r}(E_{P_i})\,,
\end{equation}
where for $P$ an ordered tableau with entries $\le m$ we define $E_P=E_{(P^{-1}(1),\ldots,P^{-1}(m))}$
and for an $m$-tuple $(D_1,\ldots,D_m)$ of skew Young diagrams 
$$
\vspace{1cm}{\begin{xy}
	(-26,-7)*={\begin{ytableau}\none\\ 
	\none[E_{(D_1,\ldots,D_m)}=\qquad\qquad\ ]\\
	\none\end{ytableau}},
(-9.9,-8)*={\begin{ytableau}
\none&\none&\none[\quad D_1]\\
\none&\none[\iddots]&\none\\
\none[D_m\quad]&\none&\none
\end{ytableau}}
\end{xy}}
$$
where each row or column contains boxes from at most one skew tableau $D_j$.
Now we apply \cite[Thm.~1.5]{Kou} and we can refine our previous filtration to a filtration with sections $$\nabla_{\GL_r}(\mu^i)\ot\nabla_{\GL_s}(\lambda^i)\,.$$ Here the $\lambda^i$ are again suitably redefined and the $\mu^i$ have length $\le r$. Furthermore, the labelling is coming from triples $(P,\mu,\ov\alpha)$ where $P$ is an ordered tableau of weight $\nu$, $\mu$ a partition of $t$ and $\ov\alpha:\mu\to E_P$ goes through a set of admissible representatives for the special semi-standard tableaux of shape $\mu$ and weight the tuple of row lengths of $E_P$.
These triples are in one-one correspondence with the triples $(P,Q,\alpha)$ mentioned earlier.

We now have to check that our filtration is indeed given by spans of twisted bideterminants.
From Remark~\ref{rem.twisted_bidet} it is clear that under the section-isomorphism \eqref{eq.sectioniso1} the element $(S\,|^m_\alpha\, S_{\lambda^i})\ot(T\,|\,S_{\lambda^i})$, $S$ of shape $\mu$ with entries $\le r$, $\alpha:\mu\to\lambda^i$, $T$ of shape $\lambda^i$ with entries $\le s$, is mapped to $(S\,|^m_\alpha\,T)$ modulo the $(i+1)$-th filtration space. So it now suffices to show that at ``stage \eqref{eq.iso3}" the elements $(S\,|^m_\alpha\, S_{\lambda^i})$ correspond under the isomorphism \eqref{eq.iso3} combined with the section isomorphism of \cite[Thm.~II.4.11]{ABW} to the elements defining the filtration of $\nabla_{\GL_r}(E_{P_i})$ from \cite[Thm.~1.5]{Kou}.

For this we focus on one particular $i$ which we suppress in the notation. If $\alpha:\mu\to\lambda$ is an admissible representative of an $m$-tuple of special semi-standard tableaux, then the diagram mapping $\ov\alpha:\mu\to E_P$ whose restrictions $:Q^{-1}(j)\to P^{-1}(j)$ are the same as those of $\alpha$, is an admissible representative of the special semi-standard tableau $T=S_{E_P}\circ\ov\alpha$ of shape $\mu$. The elements defining the filtration of $\nabla_{\GL_r}(E_{P})$ from the proof of \cite[Thm.~1.5]{Kou} 
are $(S\,|_{\ov\alpha}\,S_{E_P})$, $S$ of shape $\mu$ with entries $\le r$. Here one should bear in mind that in \cite{Kou} the bideterminants are formed row-wise rather than column-wise, and that there $\ov\alpha^{^{-1}}$ is used rather than $\ov\alpha$: the map $f_T$ on page 93 of \cite{Kou} satisfies (after transposing) $T\circ f_T=S_{E_P}$, and it corresponds to the inverse of our $\ov\alpha$.\footnote{Actually the $\ov\alpha$ corresponding to the $f_T$ from \cite{Kou} are the (unique) special representatives of the special semi-standard tableaux $T$ of shape $\mu$ and weight the tuple of row lengths of $E_P$, but it is clear that the arguments there work for any choice of admissible representatives $\ov\alpha$. Furthermore, it is clear from the proof of Claim 2 on p 94 of \cite{Kou} that the filtration of $\nabla_{\GL_r}(E_P)$ does not depend on the choice of representing $\ov\alpha$'s.} 
By Remark~\ref{rem.twisted_bidet} we have
$$(S\,|_{\ov\alpha}\,S_{E_P})=\sum_{\pi\in X}{\sgn(\pi)}(S^{\ov\alpha,\pi}\,|\,S_{E_P})\,,$$
where $X$ is a set of representatives for the left cosets of $C_\mu\cap{\ov\alpha}^{^{-1}}C_{E_P}\ov\alpha$ in $C_\mu$ and
$S^{\ov\alpha,\pi}(a)=S(\pi(\ov\alpha^{^{-1}}(a)))$ for $a\in E_P$.
Now we have $C_\mu\cap{\ov\alpha}^{^{-1}}C_{E_P}\ov\alpha=C_\mu(Q)\cap\alpha^{-1}C_\lambda(P)\alpha$, so, by Remark~\ref{rem.twisted_bidet} we have for the same set $X$
$$(S\,|^m_\alpha\,S_\lambda)=\sum_{\pi\in X}{\sgn(\pi)}(S^{\alpha,\pi}\,|\,S_\lambda)\,,$$
where $S^{\alpha,\pi}(a)=S(\pi(\alpha^{-1}(a)))+(P(a)-1)r$ for $a\in\lambda$.
Under the isomorphism \eqref{eq.iso3} combined with the section isomorphism of \cite[Thm.~II.4.11]{ABW} $S^{\ov\alpha,\pi}$ corresponds to $S^{\alpha,\pi}$, that is,
$(S^{\ov\alpha,\pi}\,|\,S_{E_P})$ is mapped to $(S^{\alpha,\pi}\,|\,S_\lambda)$ modulo the filtration space labelled by ``the next $P$". So, by the above two equations,
$(S\,|_{\ov\alpha}\,S_{E_P})$ is mapped to $(S\,|^m_\alpha\, S_{\lambda})$ modulo the filtration space labelled by the next $P$.
\end{proof}

\begin{corgl}
Let $\lambda,\mu$ be partitions of $t$ with $l(\mu)\le r$ and $l(\lambda)\le s$ and let $\nu\in\Sigma_t$.
Then the elements $(S_\mu\,|^m_\alpha\,S_\lambda)$, $P,Q$ ordered tableaux of shapes $\lambda$ and $\mu$, both of weight $\nu$, and $\alpha$ in a set of representatives for the $m$-tuples of special semi-standard tableaux with shapes determined by $Q$ and weights determined by $P$, form a basis of the piece of degree $\nu$ of $k[\Mat_{rs}^m]^{U_r\times U_s}_{(\mu,\lambda)}$.
\end{corgl}
\begin{proof}
It is easy to see, using Remark~\ref{rem.twisted_bidet} for example, that the elements $(S_\mu\,|^m_\alpha\,S_\lambda)$ are highest weight vectors of the given weight. Furthermore, they are linearly independent by Theorem~\ref{thm.good_filtration}. On the other hand it follows from standard properties of good filtrations, see \cite[Prop.~II.4.13]{Jan}, that the dimension of $k[\Mat_{rs}^m]^{U_r\times U_s}_{(\mu,\lambda)}$ is equal to the number of sections $\nabla_{\GL_r}(\mu^i)\ot\nabla_{\GL_s}(\lambda^i)$ with $(\lambda_i,\mu_i)=(\lambda,\mu)$ in a good filtration of $k[\Mat_{rs}^m]$. But this is equal to the number of elements of our linearly independent set.
\end{proof}

Finally we give a version for the above corollary for the $\GL_r\times\GL_s$-action on $k[\Mat_{rs}^m]$ defined by $((A,B)\cdot f)(\un X)=f((A^{-1}X_iB)_{1\le i\le m})$, that is,
we twist the $\GL_r$-action we considered previously with the inverse transpose. We define the {\it anti-canonical tableau} $\tilde S_\mu$ of shape $\mu$ by
$\tilde S_\mu(a)=r-a_1+1$, for $a\in\mu$ where $a_1$ is the row index of $a$.
For a tuple $\mu$ of integers of length $\le r$ we denote by $\mu^{\rm rev}$ the reverse of the $r$-tuple obtained from $\mu$ by extending it with zeros.

\begin{corgl}
Let $\lambda,\mu$ be partitions of $t$ with $l(\mu)\le r$ and $l(\lambda)\le s$  and let $\nu\in\Sigma_t$.
Then the elements $(\tilde S_\mu\,|^m_\alpha\,S_\lambda)$, $P,Q$ ordered tableaux of shapes $\lambda$ and $\mu$, both of weight $\nu$,
and $\alpha$ in a set of representatives for the $m$-tuples of special semi-standard tableaux with shapes determined by $Q$ and weights determined by $P$,
form a basis of the piece of degree $\nu$ of $k[\Mat_{rs}^m]^{U_r\times U_s}_{(-\mu^{\rm rev},\lambda)}$.
\end{corgl} 

\begin{remsgl}
1.\ We now extract from the proof of Theorem~\ref{thm.good_filtration} how the triples $(P,Q,\alpha)$ are enumerated. First we order the $P$'s by identifying each $P$ with the tuple of Young diagrams (i.e. partitions) $P^{-1}(\{1,\ldots,m-i\})_{0\le i\le m-1}$ and ordering these lexicographically, where the partitions are themselves also ordered lexicographically. For a fixed $P$ we order the pairs $(Q,\alpha)$ as follows. For each $i$ we let $S_i$ be the tableau obtained by shifting the entries of $S_{P^{-1}(i)}\circ\alpha_i$ by $\sum_{j=0}^{i-1}r_j$, where $r_j$ is the number of rows of $P^{-1}(j)$. Here the $\alpha_i$ are defined as in Section~\ref{s.skew_diagrams_and_tableaux}. Let $S_{Q,\alpha}$ be the tableau of the same shape as $Q$ obtained by piecing the $S_i$ together according to $Q$. Then we say that $(Q^1,\alpha^1)>(Q^2,\alpha^2)$ if the standard enumeration of $S_{Q^1,\alpha^1}$ is lexicographically less than that of $S_{Q^2,\alpha^2}$. Now we order the triples $(P,Q,\alpha)$ lexicographically by first comparing the $P$-component and then the $(Q,\alpha)$-component. Finally, we enumerate the triples $(P,Q,\alpha)$ in decreasing order.\\
2.\ Let $E$ and $F$ be skew Young diagrams with $t$ boxes. In \cite[Thm.~3]{T3} a basis was given of the space of coinvariants $\big(S(E)\ot S(F)\big)_{\Sym_t}$
labelled by admissible representatives of special semi-standard tableaux of shape $F$ and weight the tuple of row lengths of $E$.
We give a characteristic free version of this result and two interpretations.

Let $r$ be $\ge$ the number of rows of $F$ and let $s$ be $\ge$ the number of rows of $E$, then the twisted bideterminants $(S_F\,|_\alpha\,S_E)\in k[\Mat_{rs}]$ where
$\alpha$ goes through a set of admissible representatives of special semi-standard tableaux of shape $F$ and weight the tuple of row lengths of $E$,
are linearly independent.

This can be deduced from \cite{Kou} as follows. Write $F=\mu/\tilde\mu$ and take $\ov E$ to be $E$ with $\tilde\mu$ above and
to the right of it in such a way that they have no rows or columns in common. We use the definition of Schur modules from Remark~\ref{rems.bidet}.1 which uses the right multiplication action.
If we combine this with Remark~\ref{rems.bidet}.2 
we obtain an isomorphism $\nabla_{\GL_s}(F)\stackrel{\sim}{\to}\nabla_{\GL_{s_1}}(\mu)^{U_{s'}}_{\tilde\mu}$ where $s_1=s'+s$.
By Remark~\ref{rem.twisted_bidet} this isomorphism maps $(S_F\,|_\alpha\,S_E)$ to $(S_\mu\,|_{\ov\alpha}\,S_{\ov E})$ where $\ov\alpha:\mu\to\ov E$ is given by $\ov\alpha|_F=\alpha$ and $\alpha|_{\tilde\mu}=\id$. 
For $\alpha$ as above, $\ov\alpha$ goes through a set of representatives for the special tableaux of shape $\mu$ and weight the tuple of row lengths of $\ov E$.
Since the elements $(S_\mu\,|_{\ov\alpha}\,S_{\ov E})$ are linearly independent by the proof of \cite[Thm.~1.5]{Kou}, the result follows.

Now we give two interpretations of this result. Firstly, the span of the above bideterminants can be seen as $k\ot_{\mb Z}N_{\mb Z}$, where $N_{\mb Z}$ is the intersection of $\big(S_{t,\mb Q}(E)\ot S_{t,\mb Q}(F)\big)_{\Sym_t}$ with the obvious $\mb Z$-form of $\big (M_{t,\mb Q}(E)\ot M_{t,\mb Q}(F)\big)_{\Sym_t}$. Secondly, when $r\ge$ the number of rows of $F$, this span can be identified with $\Hom_{\GL_r}(\Delta_{\GL_r}(F),\nabla_{\GL_r}(E))$. Indeed we have by \cite[Thm.~1.1(g)]{Kou}
\begin{align*}
D_L(U,S_F)\cdot (S_F\,|\,T_F)&=(\boxed{U}\,|\,T_F)\,,\\
D_L(U,S_F)\cdot (S_F\,|_\alpha\,S_E)&=(\boxed{U}\,|_\alpha\,S_E)\,.
\end{align*}
Here we adapted the notation to that of our paper: We use $T_E$ and $S_E$ instead of $\overset{*\ }{T_E}$ and $T_E$. Furthermore, we associate bideterminants column-wise rather than row-wise,
so the bideterminants of shape $\lambda$ from \cite{Kou} have shape $\lambda'$ in our notation, $\boxed{U}$ means sum over all tableaux {\it row} equivalent to $U$ etc.
So the module $\Delta_{\GL_r}(F)$ is cyclic generated by $(S_F\,|\,T_F)$ and the homomorphisms 
$$(\boxed{U}\,|\,T_F)\mapsto(\boxed{U}\,|_\alpha\,S_E):\Delta_{\GL_r}(F)\to\nabla_{\GL_r}(E)$$
are linearly independent, since their images of the generator $(S_F\,|\,T_F)$ are linearly independent by the above result. We have seen in Section~\ref{s.bidets_skew_Schur_Specht} that their number is equal to the dimension of $\Hom_{\GL_r}(\Delta_{\GL_r}(F),\nabla_{\GL_r}(E))$, so they must form a basis. In general the above homomorphisms will always span $\Hom_{\GL_r}(\Delta_{\GL_r}(F),\nabla_{\GL_r}(E))$ by \cite[Prop.~1.5(i)]{Don2} applied to $\lambda=t\varepsilon_1$, $G=\GL_{r'}$ for some $r'\ge$ the number of rows of $F$ and $\Sigma$ the root system of $\GL_r$. 
For dimension reasons, see Section~\ref{s.bidets_skew_Schur_Specht}, these homomorphisms will then also form a basis when $r\ge$ the number of rows of $E$.\\
3.\ Assume $r=r_1+\cdots+r_m$ for certain integers $r_j>0$.
By similar arguments as in the proof of Theorem~\ref{thm.good_filtration} one can construct a ``good" $\big(\prod_{j=1}^m\GL_{r_j}\big)\times\GL_s$-filtration
of the degree $\nu$ piece of $k[\Mat_{rs}]$  using a spanning set labelled by triples $\big(\lambda,(\mu^1,\ldots,\mu^m),\alpha\big)$, where $\lambda$ is a
partition of $t=|\nu|$ of length $\le s$, $(\mu^1,\ldots,\mu^m)$ is an $m$-tuple of partitions with $\mu^j$ of length $\le r_j$ and $|\mu^1|+\cdots+|\mu^m|=t$, 
and where $\alpha:E_{(\mu^1,\ldots,\mu^m)}\to\lambda$ goes through a set of admissible representatives for the special semi-standard tableaux of
shape $E_{(\mu^1,\ldots,\mu^m)}$ and weight $\lambda$.
These triples are in one-one correspondence with the triples $(P,(\mu^1,\ldots,\mu^m),(\alpha_1,\ldots,\alpha_m))$, where
$P$ is an ordered tableau of weight $\nu$, $\big(\mu^1,\ldots,\mu^m\big)$ is an $m$-tuple of partitions
with $\mu_j$ of length $\le r_j$ and $|\mu^1|+\cdots+|\mu^m|=t=|\nu|$, and each $\alpha_j:\mu_j\to P^{-1}(j)$ goes through a set of admissible representatives for the
special semi-standard tableaux of shape $\mu_j$ and weight the tuple of row lengths of $P^{-1}(j)$.
The filtration spaces are spanned by twisted bideterminants $(S\,|_\alpha\, T)$, where $S$ is of shape $E_{(\mu^1,\ldots,\mu^m)}$ with entries $\le r$, satisfying $S^{-1}\big((\sum_{l=1}^{j-1}r_l+\{1,\ldots,r_j\}\big)=\mu^j\subseteq E_{(\mu^1,\ldots,\mu^m)}$ for all $j$, $T$ is of shape $\lambda$ with entries $\le s$ and $\alpha:E_{(\mu^1,\ldots,\mu^m)}\to\lambda$ is as above.
\end{remsgl}

\section{Highest weight vectors for the conjugation action of $\GL_n$ on polynomials}\label{s.conjugation_action}
Firstly, let us introduce some further notation. For $n$ a natural number and $\lambda$, $\mu$ partitions with $l(\lambda)+l(\mu)\le n$, define the descending $n$-tuple
$$[\lambda,\mu]:=(\lambda_1,\ldots,\lambda_{l(\lambda)},0,\ldots,0,-\mu_{l(\mu)},\ldots,-\mu_1).$$

The group $\GL_n$ acts on $\Mat_n$ via the conjugation action, given by $S\cdot A=SAS^{-1}$ and therefore on the coordinate ring $k[\Mat_n]$ via $(S\cdot f)(A)=f(S^{-1}AS)$
Note that the nilpotent cone $\mc N_n=\{A\in\Mat_n\,|\,A^n=0\}$ is under this action a $\GL_n$-stable closed subvariety of $\Mat_n$.
We denote the algebra of invariants of $k[\Mat_n]$ under the conjugation action by $k[\Mat_n]^{\GL_n}$.
It is well-known that this is the polynomial algebra in the traces of the exterior powers of the matrix.

Now let $r,s$ be integers $\ge0$ with $r+s\le n$. We let $\GL_r\times\GL_s$ act on $k[\Mat_{rs}^m]$ as at the end of Section~\ref{s.several_matrices}:
we use the inverse rather than the transpose to define the action of $\GL_r$.
For a matrix $M$ denote by $M_{r\rfloor\lfloor s}$ the lower left $r\times s$ corner of $M$.
For $m$ an integer $\ge 2$ we define the map $\varphi_{r,s,n,m}:\Mat_n\to\Mat_{rs}^m$ by
\begin{align*}
\varphi_{r,s,n,m}(X)=\Big(X_{r\rfloor\lfloor s}\,,(X^2)_{r\rfloor\lfloor s}\,,\ldots,(X^m)_{r\rfloor\lfloor s}\Big)\,.
\end{align*}
The restriction of this map to the nilpotent cone $\mc N_n$ will be denoted by the same symbol.
In \cite{T3} the following result was proved.
\begin{thmgl}[{\cite[Thm.~1]{T3}}]\label{thm.surjective_pullback}
Let $\chi=[\lambda,\mu]$ be a dominant weight in the root lattice, $l(\mu)\le r$, $l(\lambda)\le s$, $r+s\le n$.
Then the pull-back map $$k[\Mat_{rs}^{n-1}]^{U_r\times U_s}_{(-\mu^{\rm rev},\lambda)}\to k[\mc N_n]^{U_n}_\chi$$
along $\varphi_{r,s,n,n-1}:\mc N_n\to\Mat_{rs}^{n-1}$ is surjective.
\end{thmgl}

Combining this with Corollary~2 to Theorem~\ref{thm.good_filtration} we obtain

\begin{thmgl}\label{thm.pullback_span_1}
Let $\chi=[\lambda,\mu]$ be a dominant weight in the root lattice, $l(\mu)\le r$, $l(\lambda)\le s$, $|\lambda|=|\mu|=t$, $r+s\le n$.
Then the pull-backs of the elements $(\tilde S_\mu\,|^m_\alpha\,S_\lambda)$, $\nu,P,Q,\alpha$ as in Cor.~2 to Thm.~\ref{thm.good_filtration}, along $\varphi_{r,s,n,n-1}:\mc N_n\to\Mat_{rs}^{n-1}$
span the vector space $k[\mc N_n]^{U_n}_\chi$.
\end{thmgl}

Next we recall the following instance of the graded Nakayama Lemma from \cite{T3}.
\begin{lemgl}[{\cite[Lem.~1]{T3}}]\label{lem.reduction_to_nilpotent_cone}
Let $f_1,\ldots,f_l\in k[\Mat_n]^{U_n}_\chi$ be homogeneous. If the restrictions $f_1|_{\mc N_n},\ldots,f_l|_{\mc N_n}$ span $k[\mc N_n]^{U_n}_\chi$,
then $f_1,\ldots,f_l$ span $k[\Mat_n]^{U_n}_\chi$ as a $k[\Mat_n]^{\GL_n}$-module.
The same holds with ``span" replaced by ``form a basis of".
\end{lemgl}

Combining Theorem~\ref{thm.pullback_span_1} and Lemma~\ref{lem.reduction_to_nilpotent_cone} we finally obtain

\begin{thmgl}\label{thm.pullback_span_2}
Let $\chi=[\lambda,\mu]$ be a dominant weight in the root lattice, $l(\mu)\le r$, $l(\lambda)\le s$, $|\lambda|=|\mu|=t$, $r+s\le n$.
Then the pull-backs of the elements $(\tilde S_\mu\,|^m_\alpha\,S_\lambda)$, $\nu,P,Q,\alpha$ as in Cor.~2 to Thm.~\ref{thm.good_filtration}, along $\varphi_{r,s,n,n-1}:\Mat_n\to\Mat_{rs}^{n-1}$
span the $k[\Mat_n]^{\GL_n}$-module $k[\Mat_n]^{U_n}_\chi$.
\end{thmgl}

\begin{remsgl}
1.\ Note that pulling the $(\tilde S_\mu\,|^m_\alpha\,S_\lambda)$ back just amounts to interpreting $x(Q(a))_{ij}$ as the $(i,j)$-th entry of the $Q(a)$-th matrix power and replacing $r-a_1+1$ by $n-a_1+1$. In particular, these pulled-back functions don't depend on the choice of $r$ and $s$.\\
2.\ One obtains a bigger, ``easier" spanning set by allowing arbitrary $P,Q$ of weight $\nu$ and arbitrary bijections $\alpha:\mu\to\lambda$ with $P\circ\alpha=Q$.
\end{remsgl}

\bigskip

{\sc\noindent School of Mathematics,\\
University of Leeds, LS2 9JT, Leeds, UK.\\
{\rm E-mail addresses : }{{\tt mmadjd@leeds.uk} \text{\rm and} {\tt R.H.Tange@leeds.ac.uk}}
}

\end{document}